\documentclass[preprint,12pt]{elsarticle}

\usepackage{enumerate}
\usepackage{amssymb}
\usepackage{lipsum}
\usepackage[a4paper, total={7.4in, 9.8in}]{geometry}
\usepackage{booktabs}
\usepackage{amsmath,amssymb,url}
\usepackage{enumitem} 
\usepackage{graphics} 
\usepackage{array}
\usepackage{dsfont}
\usepackage[all]{xy}
\usepackage{amsthm}
\usepackage{tikz}

\numberwithin{equation}{section}
\newtheorem{theorem}{Theorem}[section]
\newtheorem{lemma}[theorem]{Lemma}
\newtheorem{proposition}[theorem]{Proposition}
\newtheorem{corollary}[theorem]{Corollary}

\newtheorem{question}[theorem]{Question}

\newtheorem{definition}[theorem]{Definition}

\newtheorem{remark}[theorem]{Remark}
\newtheorem{example}[theorem]{Example}

\newcommand{\dn}{\mathord{\downarrow}\hspace{0.05em}}
\newcommand{\up}{\mathord{\uparrow}\hspace{0.05em}}
\newcommand{\uuar}{\mathord{\Downarrow}\hspace{0.05em}}

\newcommand\blfootnote[1]{%
\begingroup
\renewcommand\thefootnote{}\footnote{#1}%
\addtocounter{footnote}{-1}%
\endgroup
}

\journal{Topology and its Applications}

\begin{document}

\begin{frontmatter}



\title{When do weakly first-countable spaces and the Scott topology of open set lattice become sober?}


\author{Zhengmao He, Bin Zhao$^{\ast}$}
\address{School of Mathematics and Statistics, Shaanxi Normal University,Xi'an 710119, China}
\begin{abstract}
In this paper, we investigate the sobriety of weakly first-countable spaces and give some sufficient conditions that the Scott topologies of the open set lattices are sober. The main results are:

(1) Let $P$ and $Q$ be two posets. If $\Sigma P\times \Sigma Q$ is a Fr\'{e}chet space, then $\Sigma (P\times Q)=\Sigma P \times \Sigma Q$.

(2) For every $\omega$-well-filtered coherent $d$-space $X$, if $X\times X$ is a Fr\'{e}chet space,
 then $X$ is sober;

(3) For every $\omega$ type P-space or consonant Wilker space $X$, $\Sigma\mathcal{O}(X)$ is sober.
\end{abstract}
\begin{keyword} Fr\'{e}chet space; Scott topology; well-filtered space; sober space\\
\vspace*{0.2cm}
\end{keyword}


\end{frontmatter}
\blfootnote{*Corresponding author}
\blfootnote{This work is supported by the National Natural Science Foundation of China (No.12331016).}
\blfootnote{E-mail address: hezhengmaomath@163.com (Z.M. He), zhaobin@snnu.edu.cn (B. Zhao).}


\section{Introduction}
\label{}

\quad Sobriety is probably the most important and useful property of non-Hausdorff topological spaces. It plays an important role in Stone duality (see \cite{MS36}), the Hofmann-Mislove Theorem (see \cite{GG03}), representation of a closed lattice (see \cite{IWJ81}). A crucial topic of Domain theory form topological viewpoint mentioned in \cite{GG03} is to characterize the sobriety of Scott topological spaces. There are two routes for studying this issue. One is to give general conclusions form the positive aspect. A well known conclusion is every continuous dcpo or quasicontinuous dcpo is Scott sober (see \cite{GG03,JG13}). Particularly, if the sup operation of a sup semilattice dcpo is jointly continuous, the Scott topology is still sober. In addition, Miao, Xi, Li and Zhao proved that a complete lattice having countably incremental ideals is Scott sober(see \cite{AB75}). Under some topological requirements, there are a series of sufficient conditions. A core-compact (or first-countable) well-filtered dcpo is Scott sober(see \cite{DZ15,TH62}). Furthermore, Xu prove that if a $\ell f_{\omega}$-poset $P$ is Lawson compact (or having property $R$), then $\Sigma P$ is sober (see \cite{ME182}). Another route is to construct counterexamples. Johnstone first constructed a dcpo $\mathbb{J}$ whose Scott topology is non-sober (see \cite{PJ81}). Later, a Scott non-sober complete lattice is given by Isbell( see \cite{IJ82}). In \cite{RTT45}, A. Jung and S. Abramsky ask whether every distributive complete lattice is Scott sober. Xu, Xi and Zhao in \cite{XXZ20} prove that for Isbell's complete lattice $L$, the Scott space of the frame $\mathcal{Q}(L)$ is non sober, where $\mathcal{Q}(L)$ is the poset of all nonempty compact saturated subsets of $\Sigma L$ with inverse inclusion order. A further question asked by A. Jung is whether every countable complete lattice is Scott sober. Miao, Xi, Li and Zhao construct a countable Scott non-sober distributive complete lattice (see \cite{AB75}). In this paper, we will consider the sobriety of weakly first-countable spaces and the
Scott topologies of the open set lattices.

\quad Given two posets $P$ and $Q$, it is not difficult to check that the product topology of $\Sigma P$ and $\Sigma Q$ is contained in the Scott topology of $P\times Q$. There are serval sufficient conditions which leads to the equality of the product topology $\Sigma P\times\Sigma Q$ and the Scott topology of $P\times Q$. If $\Sigma P$ and $\Sigma Q$ are core-compact, first-countable or the all incremental ideals $|Id(P)|$ and $|Id(Q)|$ are countable, then the product topology of $\Sigma P$ and $\Sigma Q$ agrees with the Scott topology of $P\times Q$ (see \cite{AB75}). There are serval weakly first-countable spaces, such as Fr\'{e}chet spaces. In this paper, we proved that if $\Sigma P\times \Sigma Q$ is a Fr\'{e}chet space, then the product topology of $\Sigma P$ and $\Sigma Q$ agrees with the Scott topology of $P\times Q$. It is known in \cite{TH62} that every first-countable $\omega$-well filtered $d$-space is sober. Using Fr\'{e}chet spaces, we show that for every $\omega$-well-filtered coherent $d$-space, if $X\times X$ is a Fr\'{e}chet space, then $X$ is sober.

\quad In \cite{XXZ20}, Xu, Xi and Zhao give a first Scott non-sober frame which is isomorphic to the open set lattice of a topological space. This prompts us to consider the sobriety of the open set lattice endowed with Scott topology. Specifically, we proved that for $\omega$ P-spaces or consonant Wilker spaces, the open set lattices are sober for Scott topology.

\section{Preliminaries}
\label{}

Let $P$ be a poset and $A\subseteq P$. We set $$\up A=\{x\in P\mid \exists\ a\in A, a\leq x \}$$ and $$\dn A=\{x\in P\mid \exists\ a\in A, x\leq a\}.$$
For every $x\in P$, we write $\dn x$ for $\dn\{x\}$  and $\up x$ for $\up\{x\}$. A subset $D\subseteq P$ is directed if for a pair $a,b\in D$, $\up a\cap \up b \cap D\neq\emptyset$.

\quad   Given a poset $P$ and $x, y\in P$. We call that $x$ is {\em way below} $y$, in symbols $x\ll y$, iff for every directed subsets $D\subseteq P$ for which $\bigvee D$ exists, $\bigvee D\in\up y$ implies $D\cap\up x\neq\emptyset$.\ If $\uuar a=\{b\in P\mid b\ll a\}$ is  directed  and $\bigvee \uuar a=a$\ for all \ $a\in P$, we call $P$  a continuous poset.

\quad A poset $P$ is called a {\em directed complete poset} ({\em  dcpo}, for short) if every
directed subset of the poset has a supremum. A dcpo $P$ is called domain if $P$ is continuous.

\quad Let $P$ be a poset. A subset $U$ of $P$ is called {\em Scott open} (see \cite{GG03})if

 (i) $U=\up U$

 (ii) for every directed subset $D$, $\bigvee D\in U$ implies
$D\cap U\neq \emptyset$, whenever $\bigvee D$ exists.\\
All Scott open sets on $P$ form the Scott topology $\sigma(P)$. We write $(P,\sigma(P))$ as $\Sigma(P)$.

\quad Let $X$ be a topological space. A nonempty subset $F\subseteq X$ is  {\em irreducible},
if for a pair closed sets $A$, $B\subseteq X$, $F\subseteq A\cup B$ implies $F\subseteq A$ or $F\subseteq B$. A $T_{0}$ space $X$ is called to be {\em sober} if every irreducible closed set $A$ equals to $\overline{\{x\}}$ for some $x\in X$.

\quad For a $T_{0}$ space $X$, the {\em specialization order} $\leq$ on $X$ is defined by
$$x\leq y \Longleftrightarrow x\in cl(\{y\}).$$ The poset of a topological space $X$ with the specialization order is denoted by $\Omega(X)$.

\quad A subset $A$ of a topological space $X$ is {\em saturated} if $A=\up A$ in $\Omega(X)$. A topological space $X$ is {\em well-filtered} ({\em $\omega$-well-filtered}) if for each filtered (countably filtered) family $\mathcal{F}$ of compact saturated subsets of $X$ and each open set $U$ of $X$, $\bigcap \mathcal{F}\subseteq U$ implies $F\subseteq U$ for some $F\in\mathcal{F}$.  It is well-known that
every sober space is well-filtered (see \cite{GG03}).

\quad Given a topological space $X$, we use $\mathcal{O}(X)$ ($\Gamma(X)$) to denote the lattice of all open (closed, respectively) subsets of $X$ with inclusion order.

\quad A topological space $X$ is {\em coherent} if for a pair $K,Q\in\mathcal{Q}(X)$, $K\cap Q\in\mathcal{Q}(X)$. A topological space $X$ is said to be {\em locally compact} if for any $U\in \mathcal{O}(X)$ and for any $x\in U$, there is an open set $V$ and a compact subset $K$ such that $x\in V\subseteq K\subseteq U$. A topological space $X$ is said to be  {\em core-compact } if  $\mathcal{O}(X)$ is continuous. It has been known that every locally compact space is core-compact.

\quad A topological space $X$ is a {\em retract} of a topological space $Y$ if there are two continuous maps $f:X\longrightarrow Y$ and $g:Y\longrightarrow X$ such that $g\circ f=id_{X}$.

\quad Given a set $X$, we use $X^{<\omega}$ to denote the collection of finite subsets of $X$.

\section{ Fr\'{e}chet spaces and sober spaces}\label{sec:fm}

In this section, using Fr\'{e}chet spaces, we will give a sufficient condition such that the product topology of $\Sigma P$ and $\Sigma Q$ and the Scott topology of $P\times Q$ coincides for two posets $P$ and $Q$. As a corollary, for every bounded complete dcpo $P$, if $\Sigma P\times \Sigma P$ is a Fr\'{e}chet space, then $\Sigma P$ is sober. Meanwhile, we extend this conclusion to general $T_{0}$ space. Concretely, if $X$ is an $\omega$-well-filtered coherent $d$-space and $X\times X$ is a Fr\'{e}chet space, then $X$ is sober.

\subsection{The applications of Fr\'{e}chet spaces to the product of Scott topologies}

\vspace*{0.5cm}

\begin{definition}{\rm(see \cite{RE77})} A topological space $X$ is a Fr\'{e}chet space if for every subset $A$, $x\in \overline{A}$ implies there is a sequence $\{x_{n}\}_{n\in\omega}$ contained in $A$ converging to $x$.
\end{definition}

It is known that every first-countable space is a Fr\'{e}chet space(see \cite{RE77}). The following example illustrates the Fr\'{e}chet space is not first-countable in general.

\begin{example}{\rm Let $X=\mathbb{R}$ be the real line and $Y=\{y_{0}\}\cup (X\setminus\mathbb{N})$, where $\mathbb{N}$ is collection of all positive natural numbers and $y_{0}\not\in X$. Define a map $f: X\longrightarrow Y$ by
$$\forall \ x\in X, \ f(x)=\left\{
             \begin{array}{ll}
              x, &\ \ x\in X\setminus\mathbb{N}, \\
              y_{0}, &\ \ x\in\mathbb{N}.
             \end{array}
           \right.$$
The closed sets of space $Y$ is generated by the family $$\{A\subseteq Y\mid f^{-1}(A)\ \mbox{is closed in}\ X\}.$$ Then $Y$ is a Fr\'{e}chet space (see \cite{RE77} Example 1.6.18). However $Y$ has no countable bases at $y_{0}$. Hence, $Y$ is not first-countable (see \cite{RE77} Example 1.4.17 ).

}
\end{example}

Interestingly, J.Goubault-Larrecq find the following result in his blog which is due to Matthew de Brecht.

\begin{theorem} {\rm (see \cite{AU41})} Let $P$ and $Q$ be two posets. If $\Sigma P$ and $\Sigma Q$ are first-countable, then $\Sigma (P\times Q)=\Sigma P \times \Sigma Q$.

\end{theorem}

Now, we will give a weaker condition than first-countability such that the Theorem 3.4 still holds.

\begin{theorem} Let $P$ and $Q$ be two posets. If $\Sigma P\times \Sigma Q$ is a Fr\'{e}chet space, then $\Sigma (P\times Q)=\Sigma P \times \Sigma Q$.

\end{theorem}

\begin{proof}
It is no doubt that every open set of the product space $\Sigma P \times \Sigma Q$ is Scott open in $P\times Q$. Let $T\in\sigma(P\times Q)$ and $(p,q)\in T$. Let $\mathcal{U}_{p}$ and $\mathcal{U}_{q}$ be the set of all Scott open neighbourhoods at $p$ and $q$ in $\Sigma P$ and $\Sigma Q$, respectively. Then $\{U\times V\mid U\in\mathcal{U}_{p},V\in\mathcal{U}_{q}\}$ forms a neighbourhood base at $(p,q)$ in $\Sigma P \times \Sigma Q$.
Assume that $T$ is not open in $\Sigma P \times \Sigma Q$. This implies $U\times V\not\subseteq T$ for all $U\in\mathcal{U}_{p},V\in\mathcal{U}_{q}$. Choose a $(p_{U,V},q_{U,V})\in U\times V\setminus T$. It is no doubt that $$(p,q)\in\overline{\{(p_{U,V},q_{U,V})\mid U\in\mathcal{U}_{p},V\in\mathcal{U}_{q}\}}_{\Sigma P\times\Sigma Q}.$$
Since $\Sigma P\times \Sigma Q$ is a Fr\'{e}chet space, there
 is a sequence $(p_{n},q_{n})_{n\in\omega}$ converging to $(p,q)$ in the product topology $\Sigma P\times \Sigma Q$, where the sequence $(p_{n},q_{n})_{n\in\omega}$ is contained in $\{(p_{U,V},q_{U,V})\mid U\in\mathcal{U}_{p},V\in\mathcal{U}_{q}\}$. This means that $(p_{n})_{n\in\omega}$ and $(q_{n})_{n\in\omega}$ converge to $p$ and $q$ in $\Sigma P$ and $\Sigma Q$, respectively. Define a map $f:P\longrightarrow \sigma(Q)$ by
$$\forall\ a\in P, f(a)=\{b\mid (a,b)\in T\}.$$
As $T$ is Scott open in $P\times Q$, $f(a)$ is a Scott open set of $Q$. So $f$ is well-defined. Let $D$ be a directed subset of $P$ with $\bigvee D$ existing. Then we have $$\bigcup f(D)\subseteq f(\bigvee D)=\{b\mid (\bigvee D,b)\in T\}.$$ Choose a $b\in f(\bigvee D)$. Since $(\bigvee D,b)=\bigvee\limits_{d\in D}(d,b)\in T\in\sigma(P\times Q)$, $(d_{0},b)\in T$ for some $d_{0}\in D$. Equivalently, $b\in f(d_{0})\subseteq\bigcup f(D)$. So $\bigcup f(D)=f(\bigvee D)$ and thus $f$ is Scott continuous. Since $(q_{n})_{n\in\omega}$ converges to $q$ and $q\in f(p)\in\sigma(Q)$, $$\{q_{n}\mid n_{0}\leq n\}\cup\{q\}\subseteq f(p)$$ for some $n_{0}\in\omega$. Let $$E=\{q_{n}\mid n_{0}\leq n\}\cup\{q\}$$ and $Q\subseteq \bigcup\limits_{j\in J} V_{j}$, for some $\{V_{j}\mid j\in J\}\subseteq \sigma(Q)$. Choose a $q\in V_{j_{0}}$. Then then is a $n_{1}$ with $n_{0}\leq n_{1}$ such that $q_{m}\in V_{j_{0}}$ for all $n_{1}\leq m$. $\forall\ n_{0}\leq k\leq n_{1}$, take a $V_{j_{k}}$ satisfying $q_{k}\in V_{j_{k}}$. Then we have that $$\{V_{j_{k}}\mid n_{0}\leq k\leq n_{1}\}\cup\{V_{j_{0}}\}$$ is a finite open cover of $E$. Therefore, $E$ is a compact subset of $\Sigma Q$. It follows that $\Box E=\{G\in\sigma(Q)\mid E\subseteq G\}$ is Scott open of $\sigma(Q)$. Obviously, $f(p)\in\Box E$. By the continuity of $f$, $f^{-1}(\Box E)$ is a Scott open neighbourhood of $p$. As the sequence $(p_{n})_{n\in\omega}$ converges to $p$, there exists $n_{2}\in\omega$ with $n_{0}\leq n_{2}$ such that $p_{n}\in f^{-1}(\Box E)$ for all $n_{2}\leq n$. Particularly, $p_{n_{2}}\in f^{-1}(\Box E)$, or equivalently, $E\subseteq f(p_{n_{2}})$. Consequently, $q_{n_{2}}\in f(p_{n_{2}})$ and then $(p_{n_{2}},q_{n_{2}})\in T$, a contradiction. The proof is complete.
\end{proof}

\begin{theorem} {\rm (see \cite{ME182}) Let $P$ be a poset such that $\Sigma P\times\Sigma P=\Sigma (P\times P)$. If one of the following statements hold

(1) $\Lambda P$ is compact;

(2)  $\Sigma P$ is well-filtered and coherent,

then $\Sigma P$ is sober.}
\end{theorem}

\begin{corollary}{\rm Let $P$ be a poset such that $\Sigma P\times\Sigma P$ is a Fr\'{e}chet space. If one of the following statements hold

(1) $\Lambda P$ is compact;

(2)  $\Sigma P$ is well-filtered and coherent,

then $\Sigma P$ is sober.}

\end{corollary}

\begin{corollary} Let $P$ be a bounded complete dcpo. If $\Sigma P\times\Sigma P$ is a Fr\'{e}chet space , then $\Sigma P$ is sober.

\end{corollary}

\begin{corollary}  {\rm(see \cite{TH62})} Let $P$ be a complete lattice. If $\Sigma P$ is first-countable, then $\Sigma P$ is sober .

\end{corollary}

\subsection{The applications of Fr\'{e}chet spaces to the sober spaces}

\vspace*{0.5cm}

Now, we will extend the above Corollary 3.7 to $T_{0}$ spaces. The following Theorem 3.9 is one of curial conclusions in the whole paper.

\begin{theorem} Let $X$ be an $\omega$-well-filtered coherent $d$-space. If $X\times X$ is a Fr\'{e}chet space, then $X$ is sober.

\end{theorem}

\begin{proof}
Suppose $A$ is an irreducible subset of $X$. We first verify that $\overline{A}$ is directed. Suppose $x,y\in \overline{A}$. Let $\mathcal{U}_{x}$ and $\mathcal{U}_{y}$ be the collection of all open neighbourhoods at $x$ and $y$, respectively. For every $U\in\mathcal{U}_{x}$ and $V\in\mathcal{U}_{y}$, by the irreducibility of $A$, $A\cap U\cap V\neq\emptyset$. Take a $a_{U,V}\in A\cap U\cap V$. It is not difficult to see that $\{U\times V\mid U\in\mathcal{U}_{x},V\in\mathcal{U}_{y}\}$ is a neighborhood base at $(x,y)$ in $X\times X$. This implies $$(x,y)\in\overline{\{(a_{U,V},a_{U,V})\mid U\in\mathcal{U}_{x},V\in\mathcal{U}_{y}\}}_{X\times X}.$$

Since $X\times X$ is a Fr\'{e}chet space, there is a $\{(a_{n},a_{n})\}_{n\in\omega}$ converging to $(x,y)$, where $\{(a_{n},a_{n})\}_{n\in\omega}$ is contained in $\{(a_{U,V},a_{U,V})\mid U\in\mathcal{U}_{x},V\in\mathcal{U}_{y}\}$. So $\{a_{n}\}_{n\in\omega}$ converges to $x$ and $y$ in $X$, simultaneously. Let $$Q_{n}=\up\{a_{n}\mid m\geq n\}\cup\up x$$ and $$P_{n}=\up\{a_{n}\mid m\geq n\}\cup\up y,$$ for all $n\in\omega$.\\
$\mathbf{Claim}$ 1: $Q_{k}$ and $P_{k}$ are compact in $X$, for all $k\in\omega$.

Let $Q_{k}\subseteq\bigcup\{V_{i}\mid i\in I\}$ for some open sets $\{V_{i}\mid i\in I\}$ of $X$. Take a $x\in V_{i_{0}}$. As $a_{n}$ converges to $x$, there is a $k< m_{0}$ such that $a_{n}\in V_{i_{0}}$ for all $m_{0}\leq n$. For all $k\leq s\leq m_{0}-1$, choose a $a_{s}\in V_{i_{s}}$. Then $$\{V_{r}\mid k\leq r\leq m_{0}-1,\mbox{or}\ r=i_{0}\}$$ forms a finite subcover of $Q_{k}$. Showing that $Q_{k}$ is compact and similar for $P_{k}$.

$\mathbf{Claim}$ 2: $\{Q_{k}\cap P_{k}\mid k\in\omega\}$ is a filtered family  of compact saturated subsets.

As $X$ is coherent and by claim 1, $Q_{k}\cap P_{k}$ is compact for all $n\in\omega$. Obviously, $\{Q_{k}\cap P_{k}\mid k\in\omega\}$ is filtered.

$\mathbf{Claim}$ 3: $\bigcap\limits_{k\in\omega}(Q_{k}\cap P_{k})\cap \overline{A}\neq\emptyset$.

Assume that $\bigcap\limits_{k\in\omega}(Q_{k}\cap P_{k})\cap \overline{A}=\emptyset$. Equivalently, $\bigcap\limits_{k\in\omega}(Q_{k}\cap P_{k})\subseteq X\setminus \overline{A}$. Since $X$ is $\omega$-well-filtered, $Q_{k_{0}}\cap P_{k_{0}}\subseteq X\setminus \overline{A}$ for some $k_{0}\in\omega$. Whence $$a_{k_{0}}\in Q_{k_{0}}\cap P_{k_{0}}\subseteq X\setminus \overline{A},$$ a contradiction. So claim 3 is proved.

Take a $z\in\bigcap\limits_{k\in\omega}(Q_{k}\cap P_{k})\cap \overline{A}$. Then for every $k\in\omega$, $$z\in Q_{k}\cap P_{k}=\up\{a_{n}\mid n\geq k\}\cup(\up x\cap \up y),$$ and $z\in \overline{A}$. If $z\in\up x\cap \up y$, then $z$ is an upper bound of $\{x,y\}$ in $\overline{A}$. If $z\not\in\up x\cap \up y$, then for all $k\in \omega$ there is $k\leq k_{r}$ such that $a_{k_{r}}\leq z$. Since $(a_{\alpha_{n}})_{n\in\omega}$ converges to $x$ and $y$ and $$x\in X\setminus\dn z\in\mathcal{O}(X)\ \mbox{or}\ \ y\in X\setminus\dn z\in\mathcal{O}(X),$$ there are $u\in\omega$ such that $a_{s}\in X\setminus\dn z$ for all $u\leq s$. Particular, $a_{u_{r}}\in X\setminus\dn z$, a contradiction.
Therefore, $\overline{A}$ is directed. Since $X$ is a $d$-space, $\overline{A}$ is a Scott closed set of $\Omega(X)$ and $\bigvee \overline{A}$ exists. Consequently, $\overline{A}=\overline{\{\bigvee \overline{A}\}}$ and hence $X$ is sober.
\end{proof}

\begin{corollary} Let $X$ be a well-filtered coherent space. If $X\times X$ is a Fr\'{e}chet space, then $X$ is sober.
\end{corollary}

\begin{remark} {\rm Note that for every first-countable $X$, $X\times X$ is still first-countable. Hence $X\times X$ is a Fr\'{e}chet space.}

\end{remark}

By corollary 3.11 and Remark 3.12, we have the following corollary 3.12.

\begin{corollary} Let $X$ be a  well-filtered coherent space. If $X$ is first-countable, then $X$ is sober.
\end{corollary}

It is known that every fist-countable well-filtered space is sober (see \cite{TH62}). So the above Corollary 3,12 is surely obtained.

\begin{proposition}\  Let\ $Y$\ be a\ Fr\'{e}chet\ space. If $X$ is a retract of $Y$, then $X$ is a Fr\'{e}chet space.

\end{proposition}

\begin{proof}
Suppose\ $x\in \overline{A}$, $A\subseteq X$. Since $X$ is a retract of $Y$,\ there are continuous maps $f:X\longrightarrow Y$ and
 $g:Y\longrightarrow X$ such that $g\circ f=id_{X}$. As $f$ is continuous, $\overline{f(\overline{A})}=\overline{f(A)}$. So we have $f(x)\in \overline{f(A)}$. Since $Y$ is a Fr\'{e}chet space,\ there is a sequence $\{x_{n}\}_{n\in\omega}$ contained in $A$ satisfying $\{f(x_{n})\}_{n\in\omega}$\ converging to $f(x)$.\ Take a $W\in\mathcal{O}(X)$ and $x\in W$.\ By the continuity of $g$, $f(x)\in g^{-1}(W)\in\mathcal{O}(Y)$.\ It follows form\ $\{f(x_{n})\}_{n\in\omega}$ converging to $f(x)$ that there is a $m\in\omega$  such that $f(x_{k})\in g^{-1}(W)$ for all $m\leq k$.\ Equivalently, $g(f(x_{k}))=x_{k}\in W$. This means that the sequence $\{x_{n}\}_{n\in\omega}$ converges to $x$. Proving that $X$ is a Fr\'{e}chet space.
 \end{proof}

 \begin{example} {\rm (1) Let $\mathcal{L}=\mathbb{N}\times\mathbb{N}\times(\mathbb{N}\cup\{\infty\})$, where $\mathbb{N}$ is the poset of all positive number with natural order. The partial order $\leq$ on $\mathcal{L}$ is defined by:

(i) $(m,n,1)\leq(m,n,2)\leq\cdot\cdot\cdot\leq(m,n,r)\leq(m,n,r+1)\leq\cdot\cdot\cdot\leq(m,n,\infty)$, for all $m,n,r\in\mathbb{N}$;

(ii) $(m,n,r)\leq(m+1,r+s-1,\infty)$, for all $m,n,s\in\mathbb{N}$.

\begin{center}
\centering
\includegraphics[totalheight=2.1in]{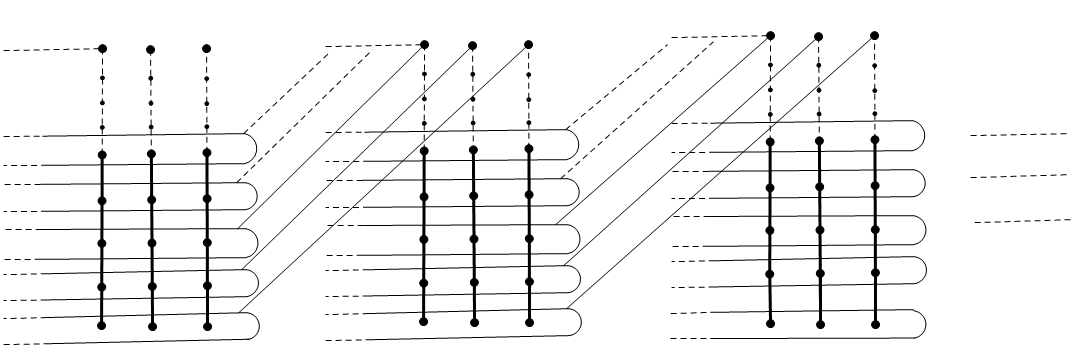}\\ The dcpo $\mathcal{L}$
\end{center}

Then $\Sigma \mathcal{L}$ is well-filtered but not sober(see \cite{E2}). Take $A=\{(2,m,n)\mid m,n\in\mathbb{N}\}$. One can check that $$\overline{A}_{\Sigma \mathcal{L}}=\{(1,m,n)\mid m\in\mathbb{N},n\in\mathbb{N}\cup\{\infty\}\}\cup\{(2,m,n)\mid m\in\mathbb{N},n\in\mathbb{N}\cup\{\infty\}\}.$$
Clearly, $(1,1,\infty)\in \overline{A}$. Suppose the sequence $\{(2,a_{n},b_{n})\}_{n\in\omega}\subseteq A$ converges to $(1,1,\infty)\in \overline{A}$. There are only 2 cases.

Case 1: There exists a $m_{0}\in\mathbb{N}$ such that $\{(2,m_{0},k)\mid k\in\mathbb{N}\}\cap\{(2,a_{n},b_{n})\mid n\in\omega\}$ is infinite.

Consider $W=\up(1,1, m_{0}+1)\cup(\bigcup\limits_{m_{0}+1\leq k}\up(2,k,1))\cup\{(a,b,c)\mid 2\leq a,(a,b,c)\in\mathcal{L}\}$. Then there is no $N\in\mathbb{N}$ satisfying \ $\{(2,a_{n},b_{n})\mid N\leq n\}\subseteq W$, a contradiction.

Case 2: For all $r\in\mathbb{N}$ such that $\{(2,r,k)\mid k\in\mathbb{N}\}\cap\{(2,a_{n},b_{n})\mid n\in\omega\}$ is finite.

In this case, we let $M_{r}$ be the maximal number of the set $\{b_{n}\mid a_{n}=r, n\in\mathbb{N}\}$. Set
$$V=\up(1,1, 1)\cup(\bigcup\limits_{r\in \mathbb{N}}\up(2,r,M_{r}+1))\cup\{(a,b,c)\mid 2\leq a,(a,b,c)\in\mathcal{L}\}.$$ Then we have  $V\cap\{(2,a_{n},b_{n})\mid n\in\omega\}=\emptyset$, impossible.

Therefore, $\Sigma \mathcal{L}$ is not a Fr\'{e}chet space. By Proposition 3.14, the sobrification of $\Sigma \mathcal{L}\times\Sigma\mathcal{L}$ is not a Fr\'{e}chet space.

(2) Let $X=\{0\}\cup\bigcup\limits_{n\in\mathbb{N}}X_{n}$, where $X_{n}=\{\frac{1}{n}\}\cup\{\frac{1}{n}+\frac{1}{n^{2}},\frac{1}{n}+\frac{1}{n^{2}+1},\cdot\cdot\cdot\}$. Equip a topology $\tau$ on $X$ as the following:

all elements having the form $\frac{1}{m}+\frac{1}{n}$ are isolated points;

every $\frac{1}{n}$ has the neighbourhood system $$\{\{\frac{1}{n}\}\cup\{\frac{1}{n}+\frac{1}{k},\frac{1}{n}+\frac{1}{k+1},\cdot\cdot\cdot\}\mid k=n^{2},n^{2}+1,\cdot\cdot\cdot\};$$

the neighbourhood system at 0 is $$\{(\bigcup\limits_{n\in \mathbb{N}\setminus M}X_{n})\setminus F\mid M\in\mathbb{N}^{<\omega}, F\in(\bigcup\limits_{n\in \mathbb{N}\setminus M}X_{n})^{<\omega}, F\subseteq\{\frac{1}{m}+\frac{1}{m^{2}+r}\mid m,r\in\mathbb{N}\}\}.$$

Then $(X,\tau)$ is Hausdorff but not a Fr\'{e}chet space(see \cite{RE77}). Hence $X$ is sober. Unfortunately, by Proposition 3.14, $X\times X$ is still not a Fr\'{e}chet space.

}
\end{example}

By Proposition 3.14, if $X\times X$ is a Fr\'{e}chet space, then $X$ is a Fr\'{e}chet space. So The following question arises naturally.

\begin{question}{\rm Let $X$ be a $\omega$-well-filtered Fr\'{e}chet space. Is $X$ always sober? }
\end{question}

\section{The sobriety of the open set lattice }

In this section, we investigate the Scott sobriety of the open set lattice. Some sufficient conditions are given.

\begin{definition} A topological space $X$ is called a $\omega$ type space if $X$ has a subbase consisting of countable subsets of $X$.

\end{definition}

 Every topological space $X$ with $|X|\leq\omega$ is a $\omega$ type space.

\begin{proposition} {\rm (1) The subspace of every $\omega$-space is a $\omega$ type space.

(2) If $X_{1},X_{2},\cdot\cdot\cdot,X_{n}$ are $\omega$-spaces, then $\prod\limits_{1\leq k \leq n}X_{k}$ is a $\omega$ type space.

}
\end{proposition}

\begin{definition}{\rm (see \cite{MTBL21})} A topological space $X$ is called a $P$-space if the intersection of countably open sets is open.

\end{definition}

In order to give the first sufficient condition, the following Theorem 4.4 is useful.

\begin{theorem} Let $X$ and $Y$ be a pair of $\omega$ type $P$-spaces. Then $\Sigma \mathcal{O}(X)\times\Sigma \mathcal{O}(Y)=\Sigma \mathcal{O}(X)\times \mathcal{O}(Y)$.

\end{theorem}

\begin{proof}
 Obviously, the product topology of $\Sigma \mathcal{O}(X)$ and $\Sigma \mathcal{O}(Y)$ is contained in the Scott topology of $\mathcal{O}(X)\times\Sigma \mathcal{O}(Y)$. Conversely, let $\mathcal{U}$ be an open set of $\Sigma \mathcal{O}(X)\times \mathcal{O}(Y)$ and $(U,V)\in\mathcal{U}$. By assumption, we can enumerate $\mathcal{A}$ and $\mathcal{B}$ as the bases of $X$ and $Y$ such that all elements belonging to $\mathcal{A}$ and $\mathcal{B}$ are countable, respectively. Then $U=\bigcup\mathcal{A}_{1}$ and $V=\bigcup\mathcal{B}_{1}$, for some $\mathcal{A}_{1}\subseteq\mathcal{A}$ and $\mathcal{B}_{1}\subseteq\mathcal{B}$. Since $\mathcal{U}$ is Scott open in $\Sigma \mathcal{O}(X)\times \mathcal{O}(Y)$, $(\bigcup\limits_{1\leq k\leq n}A_{k},\bigcup\limits_{1\leq r\leq m}B_{r})\in\mathcal{U}$ for some $A_{1},A_{2},\cdot\cdot\cdot,A_{n}\in\mathcal{A}_{1}$ and $B_{1},B_{2},\cdot\cdot\cdot,B_{m}\in\mathcal{B}_{1}$. Let $A=\bigcup\limits_{1\leq k\leq n}A_{k}$, and $B=\bigcup\limits_{1\leq r\leq m}B_{r}$. Then $A,B$ are countable with $A\subseteq U$ and $B\subseteq V$. Suppose that $A=\{a_{1},a_{2},\cdot\cdot\cdot,a_{n},\cdot\cdot\cdot\}$ and $B=\{b_{1},b_{2},\cdot\cdot\cdot,b_{n},\cdot\cdot\cdot\}$. For every $n\in \mathbb{N}^{+}$, we take
\begin{center}
$\mathcal{A}_{n}=\{C\in\mathcal{O}(X)\mid\{a_{1},a_{2},\cdot\cdot\cdot,a_{n}\}\subseteq C\},$
\end{center}
\begin{center}
$\mathcal{B}_{n}=\{D\in\mathcal{O}(Y)\mid\{b_{1},b_{2},\cdot\cdot\cdot,b_{n}\}\subseteq D\}.$
\end{center}
Then, for every $n\in \mathbb{N}^{+}$, $\mathcal{A}_{n}$ and $\mathcal{B}_{n}$ are Scott open sets in $\Sigma \mathcal{O}(X)$ and $\Sigma \mathcal{O}(Y)$, respectively. In addition, $U\in\mathcal{A}_{n}$ and $V\in\mathcal{B}_{n}$, for every $n\in \mathbb{N}^{+}$. Assume that $\mathcal{A}_{n}\times\mathcal{B}_{n}\not\subseteq\mathcal{U}$, for every $n\in \mathbb{N}^{+}$. Choose a $(C_{n},D_{n})\in(\mathcal{A}_{n}\times\mathcal{B}_{n})\setminus\mathcal{U}$. For every $n\in \mathbb{N}^{+}$, we take
\begin{center}
$A_{n}=(\bigcap_{k\geq n}C_{k})\cap A \ \ {\rm and} \ \ B_{n}=(\bigcap_{k\geq n}D_{k})\cap B.$
\end{center}
Since $X,Y$ are $P$-spaces, $A_{n}\in\mathcal{O}(X)$ and $B_{n}\in\mathcal{O}(Y)$, for every $n\in \mathbb{N}^{+}$. Clearly, $\bigcup_{n\in \mathbb{N}^{+}}A_{n}\subseteq A$ and  $\bigcup_{n\in \mathbb{N}^{+}}B_{n}\subseteq B$. For every $n\leq k$, $a_{n}\in C_{k}\cap U$ and $b_{n}\in D_{k}\cap U$. This implies that $a_{n}\in A_{n}$ and $b_{n}\in B_{n}$. Consequently, $A=\bigcup_{n\in \mathbb{N}^{+}}A_{n}$ and $B=\bigcup_{n\in \mathbb{N}^{+}}B_{n}$. It follows from $(C_{n},D_{n})\not\in\mathcal{U}$ that $(A_{n},B_{n})\not\in\mathcal{U}$. Since $\mathcal{O}(X)\times \mathcal{O}(Y)\setminus\mathcal{U}$ is Scott closed and $(A,B)$ is a directed supremum of $\{(A_{n},B_{n})\mid n\in\mathbb{N}^{+}\}$, we have $(A, B)\in\mathcal{O}(X)\times \mathcal{O}(Y)\setminus\mathcal{U}$. It is a contradiction. Therefore, there is a $\mathcal{A}_{k}\times \mathcal{B}_{k}\subseteq\mathcal{U}$ and $(U,V)\in\mathcal{A}_{k}\times \mathcal{B}_{k}$. By the arbitrariness of $(U,V)$, $\mathcal{U}$ is open in the product topology $\Sigma \mathcal{O}(X)\times\Sigma \mathcal{O}(Y)$. Proving that $\Sigma \mathcal{O}(X)\times\Sigma \mathcal{O}(Y)=\Sigma \mathcal{O}(X)\times \mathcal{O}(Y)$.
\end{proof}

\begin{lemma} {\rm(see \cite{RE77})} Let $L$ be a dcpo and sup semilattice. If the sup operation is jonitly Soctt-contniuous, then $\Sigma L$ is sober.

\end{lemma}

\begin{corollary} Let $X$ be a $\omega$ type $P$-space. Then $\Sigma \mathcal{O}(X)$ is sober.

\end{corollary}

\begin{proof}
By Theorem 4.4, $\Sigma \mathcal{O}(X)\times\Sigma \mathcal{O}(X)=\Sigma \mathcal{O}(X)\times \mathcal{O}(X)$. Then we have the sup operation of  $\Sigma \mathcal{O}(X)$ is jonitly Soctt-contniuous. Hence, by Lemma 4.5,  $\Sigma \mathcal{O}(X)$ is sober.
\end{proof}

\begin{corollary} Let $X$ be a countable $P$-space. Then $\Sigma \mathcal{O}(X)$ is sober.

\end{corollary}

\begin{theorem} Let $P,Q$ be a poset. If $\Sigma P$ and $\Sigma Q$ are $\omega$ type spaces, then $\Sigma \sigma(P)\times\Sigma \sigma(Q)=\Sigma \sigma(P)\times \sigma(Q)$.

\end{theorem}

\begin{proof}
It suffices to check that the Scott topology of $\sigma(P)\times \sigma(Q)$ is contained in the product topology. Suppose $\mathcal{A}\in\sigma(\sigma(P)\times \sigma(Q))$ and $(U,V)\in\mathcal{A}$. Since $\Sigma P$ and $\Sigma Q$ are $\omega$ type spaces, there are $\mathcal{U}\subseteq\sigma(P)$ and $\mathcal{V}\subseteq\sigma(Q)$ such that $U=\bigcup\mathcal{U}$ and $V=\bigcup\mathcal{V}$, where every element belonging to $\mathcal{U}$ and $\mathcal{V}$ is countable. By the Scott openness of $\mathcal{A}$, there are $U_{1},U_{2},\cdot\cdot\cdot,U_{n}\in\mathcal{U}$ and $V_{1},V_{2},\cdot\cdot\cdot,V_{m}\in\mathcal{V}$ satisfying $(\widetilde{U},\widetilde{V})\in\mathcal{A}$, where $\widetilde{U}=\bigcup\limits_{1\leq k\leq n}U_{k}$ and $\widetilde{V}=\bigcup\limits_{1\leq r\leq m}V_{r}$. Whence, $\widetilde{U}$ and $\widetilde{V}$ are countable. Let $\widetilde{U}=\{u_{1},u_{2},\cdot\cdot\cdot,u_{n},\cdot\cdot\cdot\}$ and $\widetilde{V}=\{v_{1},v_{2},\cdot\cdot\cdot,v_{n},\cdot\cdot\cdot\}$. For every $n\in\mathbb{N}^{+}$, we set
 \begin{center}
$\mathcal{B}_{n}=\{B\in\sigma(P)\mid\{u_{1},u_{2},\cdot\cdot\cdot,u_{n}\}\subseteq B\},$
\end{center}
\begin{center}
$\mathcal{C}_{n}=\{C\in\sigma(Q)\mid\{v_{1},v_{2},\cdot\cdot\cdot,v_{n}\}\subseteq C\}.$
\end{center}
Clearly, $\widetilde{U}\in \mathcal{B}_{n}\in\sigma(\sigma(P))$ and $\widetilde{V}\in\mathcal{C}_{n}\in\sigma(\sigma(Q))$, for every $n\in\mathbb{N}^{+}$. Assume that $\mathcal{B}_{n}\times\mathcal{C}_{n}\not\subseteq\mathcal{A}$, for all $n\in\mathbb{N}^{+}$. Choose a $(B_{n},C_{n})\in\mathcal{B}_{n}\times\mathcal{C}_{n}\setminus\mathcal{A}$. For every $n\in\mathbb{N}^{+}$, let
\begin{center}
$E_{n}=(\bigcap_{k\geq n}B_{k})\cap \widetilde{U}\ \ {\rm and} \ \ F_{n}=(\bigcap_{k\geq n}C_{k})\cap \widetilde{V}.$
\end{center}
We assert that $E_{n}\in\sigma(P)$ and $F_{n}\in\sigma(Q)$, for all $n\in\mathbb{N}^{+}$. It is no doubt that $E_{n}$ is upper. Let $D\subseteq P$ be a directed subset with $\bigvee D$ existing such that $\bigvee D\in E_{n}$. It follows form $\widetilde{U}\in\sigma(P)$ that $D\cap\widetilde{U}\neq\emptyset$. Choose a $u_{s}\in U$. If $s\leq n$, then $u_{s}\in E_{n}\cap D$. If $s>n$, then $u_{s}\in(\bigcap\limits_{s\leq k}B_{k})\cap \widetilde{U}$. As $\bigvee D\in\bigcap\limits_{n\leq k\leq s-1}B_{k}\in\sigma(P)$, $D\cap(\bigcap\limits_{n\leq k\leq s-1}B_{k})\neq\emptyset$. Take a $d\in D\cap(\bigcap\limits_{n\leq k\leq s-1}B_{k})$. Using the directedness of $D$, we can fix a $\widetilde{d}\in D\cap\up d\cap \up u_{s}$. Since every Scott open set is upper, we have
 $$\widetilde{d}\in D\cap(\bigcap\limits_{n\leq k\leq s-1}B_{k})\cap(\bigcap\limits_{s\leq k}B_{k})\cap \widetilde{U}=D\cap E_{n}.$$ Proving that $E_{n}\in\sigma(P)$. Similarly, we can verify that $F_{n}\in\sigma(Q)$. As $(B_{n},C_{n})\not\in\mathcal{A}$ and $\mathcal{A}\in\sigma(\sigma(P)\times \sigma(Q))$, $(E_{n},F_{n})\not\in\mathcal{A}$. One can check that $\mathcal{E}$ is directed in $\sigma(P)\times \sigma(Q)$ and $$\bigcup\mathcal{E}=(\bigcup\limits_{n\in\mathbb{N}^{+}}E_{n},\bigcup\limits_{n\in\mathbb{N}^{+}}F_{n})=(\widetilde{U},\widetilde{V}),$$
 where $\mathcal{E}=\{(E_{n},F_{n})\mid n\in\mathbb{N}^{+}\}$. According to the Scott closedness of $\sigma(P)\times \sigma(Q)\setminus\mathcal{A}$ and $\mathcal{E}\subseteq\sigma(P)\times \sigma(Q)\setminus\mathcal{A}$, we have $\bigcup\mathcal{E}=(\widetilde{U},\widetilde{V})\in\sigma(P)\times \sigma(Q)\setminus\mathcal{A}$, a contradiction. Hence, there is a $\mathcal{B}_{k_{0}}\times\mathcal{C}_{k_{0}}\subseteq\mathcal{A}$ for some $k_{0}\in\mathbb{N}^{+}$. This means that $(U,V)\in\mathcal{B}_{k_{0}}\times\mathcal{C}_{k_{0}}\subseteq\mathcal{A}$. Therefore, $\mathcal{A}$ is open in the product topology. The proof is complete.
 \end{proof}

\begin{corollary} Let $P$ be a poset. If $\Sigma P$ is a $\omega$ type space, then $\Sigma \sigma(P)$ is sober.

\end{corollary}

According to proof of Lemma 3.1 in \cite{AVM05}, we have the following Proposition 4.10.

\begin{proposition} Let $P$ be a countable poset. Then $\Sigma \sigma(P)$ is first countable.
\end{proposition}

Using corollary 4.9, we have the following corollary.

\begin{corollary} {\rm( see \cite{AVM05})} Let $P$ be a countable posets. Then $\Sigma \sigma(P)$ is sober.

\end{corollary}

\begin{lemma} {\rm Let $L$ be a complete lattice. Then we have

(1) if $\Sigma \Gamma(L)$ is sober, then $\Sigma L$ is sober(see \cite{WOM95});

(2) if $\Sigma \mathcal{Q}(L)$ is sober, then $\Sigma L$ is sober(see \cite{XXZ20}).}
\end{lemma}

\begin{lemma} {\rm (see \cite{AB75})} There is a countable complete lattice $L_{0}$ whose Scott topology is not sober.
\end{lemma}

\begin{corollary} $\Sigma \Gamma(L_{0})$ and $\Sigma \mathcal{Q}(L_{0})$ are not sober.

\end{corollary}

Finally, we show that the open set lattice of every consonant  Wilker space is sober for Scott topology.

\begin{definition}{\rm (see \cite{NPE10})} A topological space $X$ is a Wilker space if for two open sets $U_{1},U_{2}$ and a compact set $K$ with $ K\subseteq U_{1}\cup U_{2}$, there exist two compact sets $K_{1}\subseteq U_{1}$ and $K_{2}\subseteq U_{2}$ such that $K \subseteq K_{1}\cup K_{2}$.

\end{definition}

\begin{lemma}{\rm (see \cite{NPE10}) Every $T_{i}$ ($i=2,3,5$) space and every locally compact space is a Wilker space.}

\end{lemma}

\begin{definition}{\rm (see \cite{MAYU61})} A $T_{0}$ space $X$ is consonant if for every $\mathcal{F}\in\sigma(\mathcal{O}(X))$ and $U\in\mathcal{F}$, there is $Q\in\mathcal{Q}(X)$ such that $U\in \Phi(Q)\subseteq \mathcal{F}$, where $\Phi(Q)=\{U\in\mathcal{O}(X)\mid Q\subseteq U\}$.
\end{definition}

\begin{theorem} Let $X$ be a consonant and Wilker space. Then $\Sigma \mathcal{O}(X)$ is sober.
\end{theorem}
\begin{proof}
 Let $\mathcal{A}$ be a Scott irreducible closed subset in $\Sigma \mathcal{O}(X)$. Suppose that $\bigcup\mathcal{A}\not\in\mathcal{A}$. As $X$ is a consonant space and $\bigcup\mathcal{A}\in\mathcal{O}(X)\setminus\mathcal{A}\in\sigma(\mathcal{O}(X))$, there is a compact subset $K$ such that $\bigcup\mathcal{A}\in\Phi(K)\subseteq\mathcal{O}(X)\setminus\mathcal{A}$. It follows from $\bigcup\mathcal{A}\in\Phi(K)$ that there are $U_{1},U_{2},\cdot\cdot\cdot,U_{n}\in\mathcal{A}$ satisfying $K\subseteq U_{1}\cup U_{2}\cup\cdot\cdot\cdot\cup U_{n}$. Since $X$ is a Wliker space, there are compact sets $K_{i}\subseteq U_{i}$ ($1\leq i\leq n$) such that $K\subseteq K_{1}\cup K_{2}\cup\cdot\cdot\cdot\cup K_{n}$. Now, we have $\mathcal{A}\cap \Phi(K_{i})\neq\emptyset$, for every $1\leq i\leq n$. By the irreducibility of $\mathcal{A}$, $\mathcal{A}\cap (\bigcap\limits_{1\leq i\leq n}\Phi(K_{i}))\neq\emptyset$. Choose a $W\in\mathcal{A}\cap (\bigcap\limits_{1\leq i\leq n}\Phi(K_{i}))\neq\emptyset$. Then $W\in\mathcal{A}$ and $K_{1}\cup K_{2}\cup\cdot\cdot\cdot\cup K_{n}\subseteq W$. This implies that $W\in\Phi(K)\cap\mathcal{A}$, a contradiction. Hence, we know $\bigcup\mathcal{A}\in\mathcal{A}$. As  $\mathcal{A}$ is lower, we have $\mathcal{A}=\dn_{\mathcal{O}(X)}\bigcup\mathcal{A}=cl(\{\bigcup\mathcal{A}\})$. Proving that $\Sigma \mathcal{O}(X)$ is sober.
 \end{proof}

\begin{lemma}{\rm(see \cite{C1})} Every Qusai-Polish space is a consonant and Wliker space.

\end{lemma}

\begin{corollary} {\rm (see \cite{C1})} Let $X$ be a Qusai-Polish space. Then $\Sigma \mathcal{O}(X)$ is sober.
\end{corollary}

\section{Sober + $T_{1}\neq$ Hausdorff }

As we all know, every Hausdorff space is always sober and $T_{1}$. In this section, we will present a countable sober $T_{1}$ and non-Hausdorff space.

\begin{example}

{\rm Let $\mathbb{N}^{+}$ be the set of all positive natural numbers. Given a infinite subset $E\subseteq\mathbb{N}^{+}$, we let $E=\{e_{n}\mid n=1,2,3,\cdot\cdot\cdot\}$ such that the sequence $\{e_{n}\}$ is strictly increasing. The operators $\mathbf{K}$ and $\mathbf{H}$ send the infinite $E$ to $\mathbf{K}(E)=\{e_{2n+1}\mid n=0,1,2,3,...,\}$ and $\mathbf{H}(E)=\{e_{2n}\mid n=1,2,3,...,\}$, respectively. Take $A=\{2n\mid n=1,2,3,...\}$, $M=\{2n+1\mid n=0,1,2,...\}$, $S_{0}=\{4n+2\mid n=1,2,3,...\}$, $S=\{2,4\}\cup S_{0}$, and $T=\{4n\mid n=2,3,4,...\}$. We set the collection
\begin{center}$\begin{array}{lll}
\mathcal{A}&=&\{A,B,S,T\}\cup\{F\mid F\subseteq \mathbb{N}^{+} \mbox{is finite}\}\\
&&\cup\{\mathbf{R}_{1}\circ\mathbf{R}_{2}\circ\cdot\cdot\cdot\circ\mathbf{R}_{n}(S_{0})\cup\{2,4\}\mid \mathbf{R}_{k}\in\{\mathbf{K},\mathbf{H}\}, 1\leq k\leq n,n=1,2,3,...\}\\
&&\cup\{\mathbf{R}_{1}\circ\mathbf{R}_{2}\circ\cdot\cdot\cdot\circ\mathbf{R}_{n}(T)\mid \mathbf{R}_{k}\in\{\mathbf{K},\mathbf{H}\}, 1\leq k\leq n, n=1,2,3,...\}\\
&&\cup\{\mathbf{R}_{1}\circ\mathbf{R}_{2}\circ\cdot\cdot\cdot\circ\mathbf{R}_{n}(M)\mid \mathbf{R}_{k}\in\{\mathbf{K},\mathbf{H}\}, 1\leq k\leq n, n=1,2,3,...\}.\\
\end{array}$\end{center}
Then the topology on $\mathbb{N}^{+}$ whose closed sets are generated by $\mathcal{A}$ is denoted by $\tau$. Now, we assert that $(\mathbb{N}^{+},\tau)$ is sober $T_{1}$ but not Hausdorff.

$\mathbf{Claim}$ 1: $(\mathbb{N}^{+},\tau)$ is $T_{1}$ but not Hausdorff.

Clearly, $(\mathbb{N}^{+},\tau)$ is $T_{1}$ since all finite subsets are closed. Note that the topology base of $\tau$ is $$\{\mathbb{N}^{+}\setminus(C_{1}\cup C_{2}\cup\cdot\cdot\cdot\cup C_{n})\mid C_{1}, C_{2},\cdot\cdot\cdot, C_{n}\in\mathcal{A}\}.$$

Suppose $2\in\mathbb{N}^{+}\setminus(A_{1}\cup A_{2}\cup\cdot\cdot\cdot\cup A_{n})$ and $4\in\mathbb{N}^{+}\setminus(B_{1}\cup B_{2}\cup\cdot\cdot\cdot\cup B_{m})$, for some $A_{1}, A_{2},\cdot\cdot\cdot,A_{n}\in\mathcal{A}$ and $B_{1}, B_{2},\cdot\cdot\cdot, B_{m} \in\mathcal{A}$. Assume that $A_{1}\cup A_{2}\cup\cdot\cdot\cdot\cup A_{n}\cup B_{1}\cup B_{2}\cup\cdot\cdot\cdot\cup B_{m}=\mathbb{N}^{+}$. Then we have $4\in (A_{1}\cup A_{2}\cup\cdot\cdot\cdot\cup A_{n})\setminus(B_{1}\cup B_{2}\cup\cdot\cdot\cdot\cup B_{m})$ and $2\in (B_{1}\cup B_{2}\cup\cdot\cdot\cdot\cup B_{m})\setminus(A_{1}\cup A_{2}\cup\cdot\cdot\cdot\cup A_{n})$. Let $\mathcal{A}_{1}=\{A_{k}\mid 1\leq k\leq n, 4\in A_{k}\}$, $\mathcal{A}_{2}=\{B_{r}\mid 1\leq r\leq m, 2\in B_{r}\}$, $\mathcal{A}_{3}=\{A_{k}\mid 1\leq k\leq n, \{2,4\}\cap A_{k}=\emptyset\}$ and $\mathcal{A}_{4}=\{B_{r}\mid 1\leq r\leq m, \{2,4\}\cap B_{r}=\emptyset\}$. Then we have the following statements:

(1) $\mathcal{A}_{1}\subseteq\{F\mid F\subseteq \mathbb{N}^{+} \mbox{is finite}\}$ and $\mathcal{A}_{2}\subseteq\{F\mid F\subseteq \mathbb{N}^{+} \mbox{is finite}\}$;

(2) $\bigcup\mathcal{A}_{3}\cup \bigcup\mathcal{A}_{4}\subseteq M\cup T$.

These two statements imply $A_{1}\cup A_{2}\cup\cdot\cdot\cdot\cup A_{n}\cup B_{1}\cup B_{2}\cup\cdot\cdot\cdot\cup B_{m}\neq\mathbb{N}^{+}$, a contradiction. Hence $(\mathbb{N}^{+}\setminus(A_{1}\cup A_{2}\cup\cdot\cdot\cdot\cup A_{n}))\cap(\mathbb{N}^{+}\setminus(B_{1}\cup B_{2}\cup\cdot\cdot\cdot\cup B_{m}))\neq\emptyset$, showing that $(\mathbb{N}^{+},\tau)$ not Hausdorff.

$\mathbf{Claim}$ 2: $(\mathbb{N}^{+},\tau)$ is sober.

$\mathbf{Conclusion}$ 1. every infinite subset $B\subseteq S$ is not irreducible. 

Assume that $B$ is irreducible. We consider the following cases. 

Case 1: $S_{0}\subseteq B$.

\setlength{\baselineskip}{1.3\baselineskip}

Then we have $B\subseteq (\{2,4\}\cup\mathbf{K}(S_{0}))\cup (\{2,4\}\cup\mathbf{H}(S_{0}))$ but neither $B\subseteq (\{2,4\}\cup\mathbf{K}(S_{0}))$
nor $B\subseteq (\{2,4\}\cup\mathbf{H}(S_{0}))$, a contradiction. 

Case 2: $S_{0}\not\subseteq B$.

(i) $\mathbf{K}(S_{0})\cap B\in 2^{\mathbf{K}(S_{0})}\setminus\{{\emptyset,\mathbf{K}(S_{0})\}}$ and $\mathbf{H}(S_{0})\subseteq B$.

In this case, $B\subseteq (\mathbf{K}(S_{0})\cup\{2,4\})\cup (\mathbf{H}(S_{0}))\cup\{2,4\})$. However, $B\not\subseteq \mathbf{K}(S_{0})\cup\{2,4\}$  and  $B\not\subseteq \mathbf{H}(S_{0}))\cup\{2,4\}$, a contradiction.

(ii) $\mathbf{H}(S_{0})\cap B\in 2^{\mathbf{H}(S_{0})}\setminus\{{\emptyset,\mathbf{H}(S_{0})\}}$ and $\mathbf{K}(S_{0})\subseteq B$.

By (i), it is still impossible.

(iii) $\mathbf{H}(S_{0})\cap B\in 2^{\mathbf{H}(S_{0})}\setminus\{{\emptyset,\mathbf{H}(S_{0})\}}$ and $\mathbf{K}(S_{0})\cap B\in 2^{\mathbf{K}(S_{0})}\setminus\{{\emptyset,\mathbf{K}(S_{0})\}}$.

In this case, $B\subseteq (\mathbf{K}(S_{0})\cup\{2,4\})\cup (\mathbf{H}(S_{0}))\cup\{2,4\})$. Unfortunately, $B\not\subseteq \mathbf{K}(S_{0})\cup\{2,4\}$  and  $B\not\subseteq \mathbf{H}(S_{0}))\cup\{2,4\}$, a contradiction.
(iv) $B\subseteq \mathbf{K}(S_{0})\cup\{2,4\}$ or $B\subseteq \mathbf{H}(S_{0})\cup\{2,4\}$.

Without loss of generality, we just consider the situation $B\subseteq \mathbf{K}(S_{0})\cup\{2,4\}$. One can check that $\mathbf{K}(S_{0})\cup\{2,4\}$ is not irreducible.  Hence we just consider case (iv-1)-(iv-4).

(iv-1)  $(\mathbf{K}(\mathbf{K}(S_{0}))\cup\{2,4\})\cap B\in 2^{(\mathbf{K}(\mathbf{K}(S_{0}))\cup\{2,4\}}\setminus\{{\emptyset,\mathbf{K}(\mathbf{K}(S_{0}))\cup\{2,4\}\}}$ and $\mathbf{H}(\mathbf{K}(S_{0}))\subseteq B$.

Then $B\subseteq(\mathbf{K}(\mathbf{K}(S_{0}))\cup\{2,4\})\cup(\mathbf{H}(\mathbf{K}(S_{0}))\cup\{2,4\})$
and $B\not\subseteq D$, for all $$D\in\{\mathbf{K}(\mathbf{K}(S_{0}))\cup\{2,4\},\mathbf{H}(\mathbf{K}(S_{0})))\cup\{2,4\}\},$$ a contradiction.

(iv-2) $(\mathbf{H}(\mathbf{K}(S_{0}))\cup\{2,4\})\cap B\in 2^{(\mathbf{H}(\mathbf{K}(S_{0}))\cup\{2,4\}}\setminus\{{\emptyset,\mathbf{H}(\mathbf{K}(S_{0}))\cup\{2,4\}\}}$ and $\mathbf{K}(\mathbf{K}(S_{0}))\subseteq B$.

By (iv-1), it is still impossible.

(iv-3) $(\mathbf{H}(\mathbf{K}(S_{0}))\cup\{2,4\})\cap B\in 2^{(\mathbf{H}(\mathbf{K}(S_{0}))\cup\{2,4\}}\setminus\{{\emptyset,\mathbf{H}(\mathbf{K}(S_{0}))\cup\{2,4\}\}}$ and\\ $(\mathbf{K}(\mathbf{K}(S_{0}))\cup\{2,4\})\cap B\in 2^{(\mathbf{K}(\mathbf{K}(S_{0}))\cup\{2,4\}}\setminus\{{\emptyset,\mathbf{K}(\mathbf{K}(S_{0}))\cup\{2,4\}\}}$.

In this case, 
$B\subseteq(\mathbf{K}(\mathbf{K}(S_{0}))\cup\{2,4\})\cup(\mathbf{H}(\mathbf{K}(S_{0}))\cup\{2,4\})$. However, $B\not\subseteq E$, for all $$E\in\{\mathbf{K}(\mathbf{K}(S_{0}))\cup\{2,4\}),\mathbf{H}(\mathbf{K}(S_{0})\cup\{2,4\}\}.$$

(iv-4) $B\subseteq(\mathbf{H}(\mathbf{K}(S_{0}))\cup\{2,4\})$ or $B\subseteq(\mathbf{K}(\mathbf{K}(S_{0}))\cup\{2,4\})$

Without loss of generality, we suppose $B\subseteq(\mathbf{H}(\mathbf{K}(S_{0}))\cup\{2,4\})$. One can check that $\mathbf{H}(\mathbf{K}(S_{0})\cup\{2,4\}$ is not irreducible.
 
Repeat this process, there are a sequence \ $\{\mathbf{R}_{n}\}_{n\in\mathbb{N}^{+}}$ such that $B\subseteq \mathbf{R}_{1}\circ\mathbf{R}_{2}\circ\cdot\cdot\cdot\circ\mathbf{R}_{m}(S_{0})\cup\{2,4\}$, for all positive natural number $m$, where $\{\mathbf{R}_{n}\mid n\in\mathbb{N}^{+}\}\subseteq\{\mathbf{K},\mathbf{H}\}$. This implies $$B\subseteq\bigcap\limits_{k\in\mathbb{N}^{+}}(\mathbf{R}_{1}\circ\mathbf{R}_{2}\circ\cdot\cdot\cdot\circ\mathbf{R}_{k}(S_{0})\cup\{2,4\})=\{2,4\}.$$ It contradicts with $B$ a infinite subset. Therefore, conclusion 1 is proved.
\ \

$\mathbf{Conclusion}$ 2. every infinite subset $G\subseteq M$ is not irreducible.

Assume that $G$ is irreducible. 

Case (1): $G=M$.

Then $G=\mathbf{K}(M)\cup\mathbf{H}(M)$ but $G\not\subseteq\mathbf{K}(M),G\not\subseteq\mathbf{H}(M)$, a contradiction.

Case (2): $G\neq M$.
 
(2.1) $\mathbf{H}(M)\subseteq G$ and $\mathbf{K}(M)\cap G\in 2^{\mathbf{K}(M)}\setminus\{\emptyset,\mathbf{K}(M)\}$.

Now, we know $B\subseteq \mathbf{K}(M)\cup \mathbf{H}(M)$ and $B\not\subseteq F$, for all $F\in\{ \mathbf{K}(M), \mathbf{H}(M))\}$.

(2.2) $\mathbf{K}(M)\subseteq G$ and $\mathbf{H}(M)\cap G\in 2^{\mathbf{H}(M)}\setminus\{\emptyset,\mathbf{H}(M)\}$.

This situation is impossible by referring (2-1).

(2.3) $\mathbf{K}(M)\cap G\in 2^{\mathbf{K}(M)}\setminus\{\emptyset,\mathbf{K}(M)\}$. and $\mathbf{H}(M)\cap G\in 2^{\mathbf{H}(M)}\setminus\{\emptyset,\mathbf{H}(M)\}$.

Then $B\subseteq\mathbf{K}(M)\cup\mathbf{H}(M))$ but $B\not\subseteq I$, for all $I\in\{\mathbf{K}(M),\mathbf{H}(M)\}.$

(2.4) $G\subseteq \mathbf{H}(M)$ or $G\subseteq \mathbf{K}(M)$.

For convenience, we suppose $G\subseteq \mathbf{H}(M)$. Clearly, $\mathbf{H}(M)$ is not irreducible. Therefore, we just consider (2.4.1)-(2.4.4).

(2.4.1) $\mathbf{H}(\mathbf{H}(M))\subseteq G$ and $\mathbf{K}(\mathbf{H}(M))\cap G\in 2^{\mathbf{K}(\mathbf{H}(M))}\setminus\{\emptyset,\mathbf{K}(\mathbf{H}(M))\}$.

In this case, $G\subseteq\mathbf{K}(\mathbf{H}(M))\cup\mathbf{H}(\mathbf{H}(M))$ but $G\not\subseteq J$, for all\\
 $$J\in\{\mathbf{K}(\mathbf{H}(M)),\mathbf{H}(\mathbf{H}(M))\}.$$

(2.4.2) $\mathbf{K}(\mathbf{H}(M))\subseteq G$ and $\mathbf{H}(\mathbf{H}(M))\cap G\in 2^{\mathbf{H}(\mathbf{H}(M))}\setminus\{\emptyset,\mathbf{H}(\mathbf{H}(M))\}$.

This case is also impossible by using the method in (2.4.1).

(2.4.3) $\mathbf{K}(\mathbf{H}(M))\cap G\in 2^{\mathbf{K}(\mathbf{H}(M))}\setminus\{\emptyset,\mathbf{K}(\mathbf{H}(M))\}$ and $\mathbf{H}(\mathbf{H}(M))\cap G\in 2^{\mathbf{H}(\mathbf{H}(M))}\setminus\{\emptyset,\mathbf{H}(\mathbf{H}(M))\}$.

It is not difficult to find that $G\subseteq\mathbf{K}(\mathbf{H}(M))\cup\mathbf{H}(\mathbf{H}(M))$. Regrettably, $G\not\subseteq L$, for all
$$L\in\{\mathbf{K}(\mathbf{H}(M)),\mathbf{H}(\mathbf{H}(M))\}.$$

(2.4.4) $G\subseteq \mathbf{K}(\mathbf{H}(M))$ or $G\subseteq \mathbf{H}(\mathbf{H}(M))$.

We suppose $G\subseteq \mathbf{K}(\mathbf{H}(M))$.  

Repeat this process, there are a sequence \ $\{\mathbf{S}_{n}\}_{n\in\mathbb{N}^{+}}$ satistying $B\subseteq \mathbf{S}_{1}\circ\mathbf{S}_{2}\circ\cdot\cdot\cdot\circ\mathbf{S}_{m}(M)$, for all positive natural number $m$, where $\{\mathbf{S}_{n}\mid n\in\mathbb{N}^{+}\}\subseteq\{\mathbf{K},\mathbf{H}\}$. So we have $$B\subseteq\bigcap\limits_{k\in\mathbb{N}^{+}}(\mathbf{S}_{1}\circ\mathbf{S}_{2}\circ\cdot\cdot\cdot\circ\mathbf{S}_{k}(M)=\emptyset.$$ It contradicts with $G$ a infinite subset. Thus, we proved conclusion 2.

$\mathbf{Conclusion}$ 3. every infinite subset $H\subseteq T$ is not irreducible.

The proof is similar to conclusion 2.

$\mathbf{Conclusion}$ 4. every irreducible subset in $(\mathbb{N}^{+},\tau)$ is finite.

Now, let $F\subseteq\mathbb{N}^{+}$ be a infinite subset.
Take $F_{1}=F\cap S$, $F_{2}=F\cap B$ and $F_{3}=F\cap T$. Then one of $F_{1}$, $F_{1}$,$F_{3}$ is infinite.

(a) $F_{1}$ is infinite.

By conclusion 1, $F_{1}$ is not irreducible. Then there are $W_{1},W_{2}\in\mathcal{A}_{1}$ such that $F_{1}\subseteq W_{1}\cup W_{2}$ and $F_{1}\not\subseteq W_{1}$,$F_{1}\not\subseteq W_{2}$, where $$\mathcal{A}_{1}=\{\mathbf{R}_{1}\circ\mathbf{R}_{2}\circ\cdot\cdot\cdot\circ\mathbf{R}_{n}(S_{0})\cup\{2,4\}\mid \mathbf{R}_{k}\in\{\mathbf{K},\mathbf{H}\}, 1\leq k\leq n,n=1,2,3,...\}.$$

Consider $U=W_{1}\cup M\cup T$ and $V=W_{2}\cup M\cup T$. Clearly, $U,V$ are closed in $(\mathbb{N}^{+},\tau)$ and $F=F_{1}\cup F_{2}\cup F_{3}\subseteq U\cup V$. Note that $F\subseteq U$($F\subseteq V$) implies $F_{1}\subseteq W_{1}$($F_{1}\subseteq W_{2}$). So $F\not\subseteq U$ and $F\not\subseteq V$. In this case, $F$ is not irreducible.

(b) $F_{2}$ is infinite.

Using conclusion 2, $F_{2}$ can not be an irreducible subset. This means that we can enumerate  $X_{1},X_{2}\in\mathcal{A}_{2}$ such that $F_{2}\subseteq X_{1}\cup X_{2}$ and $F_{2}\not\subseteq X_{1}$,$F_{2}\not\subseteq X_{2}$, where $$\mathcal{A}_{2}=\{\mathbf{R}_{1}\circ\mathbf{R}_{2}\circ\cdot\cdot\cdot\circ\mathbf{R}_{n}(M)\mid \mathbf{R}_{k}\in\{\mathbf{K},\mathbf{H}\}, 1\leq k\leq n, n=1,2,3,...\}.$$

Take $Y=S\cup X_{1}\cup T$ and $Z=S\cup X_{2}\cup T$. Then $Y,Z$ are closed in $(\mathbb{N}^{+},\tau)$ and $F=F_{1}\cup F_{2}\cup F_{3}\subseteq Y\cup Z$. Furthermore, $F\subseteq Y$($F\subseteq Z$) implies $F_{2}\subseteq X_{1}$ ($F_{2}\subseteq X_{2}$). Whence, $F\not\subseteq U$ and $F\not\subseteq V$. Hence, $F$ is still not irreducible.

(C) $F_{3}$ is infinite.

By the proof of case (b), $F$ fails to be an irreducible subset.

Consequently, all irreducible subsets in $(\mathbb{N}^{+},\tau)$ must be finite. As $(\mathbb{N}^{+},\tau)$ is $T_{1}$, every irreducible subset in $(\mathbb{N}^{+},\tau)$ is a singleton set. Showing that $(\mathbb{N}^{+},\tau)$ is sober.}

\end{example}

\end{document}